\pdfoutput=1
\RequirePackage{ifpdf}
\ifpdf 
\documentclass[pdftex]{sigma}
\else
\documentclass{sigma}
\fi

\usepackage{fouridx}
\usepackage[mathscr]{euscript}

\numberwithin{equation}{section}
\newtheorem{Theorem}{Theorem}[section]
\newtheorem{Corollary}[Theorem]{Corollary}
\newtheorem{Lemma}[Theorem]{Lemma}
\newtheorem{Proposition}[Theorem]{Proposition}
 { \theoremstyle{definition}
\newtheorem{Remark}[Theorem]{Remark} }

\renewcommand{\Re}{\mathop {\mathrm {Re}}\nolimits}
\newcommand{\res}{\mathop {\mathrm {res}}\nolimits}

\newcommand{\hyp}[2]{\fourIdx{}{#1}{}{\!#2}F}
\newcommand{\hypc}[2]{\fourIdx{}{#1}{\C}{\!#2}F}
\newcommand{\Kc}[2]{\fourIdx{}{#1}{\C}{\!#2}{\EuScript K}}
\newcommand{\Gc}[2]{\fourIdx{}{#1}{\C}{\!#2}G}

\def\ov{\overline}

\def\wt{\widetilde}

\renewcommand{\Re}{\mathop {\mathrm {Re}}\nolimits}
\renewcommand{\Im}{\mathop {\mathrm {Im}}\nolimits}
\newcommand{\sgn}{\mathop {\mathrm {sgn}}\nolimits}

\def\bfa{\mathbf a}

\def\R {{\mathbb R }}
 \def\C {{\mathbb C }}
 \def\Z{{\mathbb Z}}

\def\N{{\mathbb N}}

\def\T{\mathbb T}

\def\cD{\EuScript D}
\def\cL{\mathscr L}
\def\cK{\EuScript K}
\def\cM{\EuScript M}

\def\kappa{\varkappa}
\def\epsilon{\varepsilon}
\def\phi{\varphi}
\def\le{\leqslant}
\def\ge{\geqslant}

\def\B{\mathrm B}

\def\vp{^{\vphantom \C}}
\def\G{\vp G}

\newcommand{\dd}[1]{\,{\rm d}\,{\overline{\overline{#1}}} }

\def\lambdA{{\boldsymbol{\lambda}}}

\def\mU{{\boldsymbol{\mu}}}
\def\sigmA{{\boldsymbol{\sigma}}}

\def\1{\boldsymbol{1}}
\def\2{\boldsymbol{2}}

\def\SL{\mathrm{SL}}

\def\SO{\mathrm{SO}}

\def\B{\mathrm{B}}

\begin{document}
\allowdisplaybreaks

\newcommand{\arXivNumber}{1910.10686}

\renewcommand{\thefootnote}{}

\renewcommand{\PaperNumber}{072}

\FirstPageHeading
	
\ShortArticleName{Barnes--Ismagilov Integrals and Hypergeometric Functions of the Complex Field}
	
\ArticleName{Barnes--Ismagilov Integrals and Hypergeometric\\ Functions of the Complex Field\footnote{This paper is a~contribution to the Special Issue on Elliptic Integrable Systems, Special Functions and Quantum Field Theory. The full collection is available at \href{https://www.emis.de/journals/SIGMA/elliptic-integrable-systems.html}{https://www.emis.de/journals/SIGMA/elliptic-integrable-systems.html}}}

\Author{Yury A.~NERETIN~$^{\dag^1\dag^2\dag^3\dag^4}$}

\AuthorNameForHeading{Yu.A.~Neretin}

\Address{$^{\dag^1}$~Wolfgang Pauli Institut, c/o Fakult\"at f\"ur Mathematik, Universit\"at Wien,\\
\hphantom{$^{\dag^1}$}~Oskar-Morgenstern-Platz 1, A-1090 Wien, Austria}
\EmailDD{\href{mailto:yurii.neretin@univie.ac.at}{yurii.neretin@univie.ac.at}}
\URLaddressDD{\url{http://mat.univie.ac.at/~neretin/}}
\Address{$^{\dag^2}$~Institute for Theoretical and Experimental Physics, Moscow, Russia}
\Address{$^{\dag^3}$~Faculty of Mechanics and Mathematics, Lomonosov Moscow State University, Russia}
\Address{$^{\dag^4}$~Institute for Information Transmission Problems, Moscow, Russia}

\ArticleDates{Received April 09, 2020, in final form July 17, 2020; Published online August 02, 2020}

\Abstract{We examine a family ${}_pG_{q}^{\mathbb C}\big[\genfrac{}{}{0pt}{}{(a)}{(b)};z\big]$ of integrals of Mellin--Barnes type over the space ${\mathbb Z}\times {\mathbb R}$, such functions~$G$ naturally arise in representation theory of the Lorentz group. We express ${}_pG_{q}^{\mathbb C}(z)$ as quadratic expressions in the generalized hypergeometric func\-tions~${}_{p}F_{q-1}$ and discuss further properties of the functions ${}_pG_{q}^{\mathbb C}(z)$.}


\Classification{33C20; 33C70; 22E43}

\renewcommand{\thefootnote}{\arabic{footnote}}
\setcounter{footnote}{0}

\section{The statements}

\subsection{Introduction}
Recall the Euler integral representation of the Gauss hypergeometric function:
\begin{equation}
\hyp{2}{1}\left[\genfrac{}{}{0pt}{}{a,b}{c};z\right]
=\frac{1}{\B(b,c-b)}\int_0^1 t^{b-1} (1-t)^{c-b-1} (1-zt)^{-a}\,{\rm d}t,
\label{eq:F21}
\end{equation}
where $a$, $b$, $c$ are complex numbers and
\begin{equation}
 \B(\alpha,\beta):=\int_0^1 t^{\alpha-1}(1-t)^{\beta-1}\,{\rm d}t= \frac{\Gamma(\alpha)\Gamma(\beta)}{\Gamma(\alpha+\beta)},
 \label{eq:beta-real}
\end{equation}
is the beta function. The hypergeometric functions $_{p+1}F_p$ admit the following inductive integral representation:
\begin{gather}
\hyp{p+1}{p}\left[\genfrac{}{}{0pt}{}{a_1,\dots,a_p,c}
{b_1,\dots,b_{p-1},d};z\right] :=
\frac{1}{\B(c,d-c)}\!\int_0^1\! t^{c-1} (1-t)^{d-c-1}
\hyp{p}{p-1}\!\left[\genfrac{}{}{0pt}{}{a_1,\dots,a_p}
{b_1,\dots,b_{p-1}};tz\right] {\rm d}t.\!\!\!\!\!
\label{eq:inductive}
\end{gather}

Let us replace the integration over the interval $[0,1]$ by the
integration over the complex plane for all integrals \eqref{eq:F21}--\eqref{eq:inductive}.
The expression $t^b$ for $t\in\C$, $b\in\C$ is ramified at $t=0$,
so the integration makes no sense.
But we can replace $t^b$ by $|t|^{2b}=t^b \ov t^b$.
It is better to consider a wider generality and to write
\[
t^{b|b'}:=t^b \ov t^{b'}.
\]
This expression is non-ramified if $b-b'\in\Z$.
Then the new beta-function equals
(see \cite[Section~II.3.7]{GGV}, \cite[Section~1.6]{GGR})\footnote{To evaluate
this integral, we can pass to polar coordinates
$t=r{\rm e}^{{\rm i}\phi}$. Integrating in $\phi$ we get a piece-wise $_2F_1$-expression
in $r$ with a singularity at $r=1$. It remains to apply the Slater theorem,
see \cite[Theorem~4.8.1, namely formula~(4.8.7)]{Sla}.
On the other hand, we can reduce \eqref{eq:beta} to the gamma-function
\eqref{eq:Gamma} in the usual way (see Jacobi's proof of the
beta-integral \eqref{eq:beta-real} in \cite[Section~I.1]{AAR}),
but a justification of a changing of integration order is unexpectedly
tedious.}
\begin{align}
\B^\C(\alpha|\alpha',\beta|\beta')&:=\frac 1\pi
\int_\C t^{\alpha-1|\alpha'-1} (1-t)^{\beta-1|\beta'-1}\, {\rm d}\Re t\,{\rm d}\Im t \notag\\
&\hphantom{:}=
\frac{\Gamma(\alpha)\Gamma(\beta)}{\Gamma(\alpha+\beta)}\cdot
\frac{\Gamma(1-\alpha'-\beta')}{\Gamma(1-\alpha')\Gamma(1-\beta')}.\label{eq:beta}
\end{align}
The new version of the Gauss hypergeometric function
(\textit{Gauss hypergeometric functions of the complex field})
also can be evaluated (see \cite[Theorem~3.9]{MN}),
the result has the form
\begin{gather}
\hypc{2}{1}\left[\genfrac{}{}{0pt}{}{a|a',b|b'}{c|c'};z\right]
=\hyp{2}{1}\left[\genfrac{}{}{0pt}{}{a,b}{c};z\right]
\hyp{2}{1}\left[\genfrac{}{}{0pt}{}{a',b'}{c'};\ov z\right] \nonumber\\
{}+
\left\{ \begin{matrix}
\text{product of}\\
\text{$\Gamma$-functions}
\end{matrix}
\right\} z^{1-c|1-c'}
\hyp{2}{1}\!\left[\genfrac{}{}{0pt}{}{a+1-c,b+1-c}{2-c};z\right]
\hyp{2}{1}\!\left[\genfrac{}{}{0pt}{}{a'+1-c',b'+1-c'}{2-c'};\ov z\right].\!\!\!\label{u}
\end{gather}

The Gauss hypergeometric functions satisfy numerous identities,
see for instance the book \textit{Higher Hypergeometric Functions}~\cite[Chapter~2]{HTF1}.
Usually such identities (and maybe all such identities) have counterparts
for~$_2F_1^\C$ (see a collection of formulas in \cite[Section~3]{MN}).

Counterparts of the Gauss hypergeometric functions were briefly mentioned
in the book by Gelfand, Graev and Vilenkin \cite[Section~II.3.7]{GGV} in 1962.
Later various formulas and identities related to hypergeometric functions
of complex fields of different levels and beta integrals of the complex field
appeared in works of numerous authors: Dotsenko and Fateev~\cite{DF},
Gelfand, Graev and Retakh~\cite{GGR},
Ismagilov \cite{Ism1,Ism2},
Derkachov and Manashov~\cite{DM1,DM2},
Bazhanov, Mangazeev and Sergeev \cite{BMS},
Derkachov, Manashov, Valinevich \cite{DMV1,DMV2},
Kels \cite{Kel1,Kel2},
Mimachi \cite{Mim},
Derkachov and Spiridonov \cite{DS},
Molchanov and Neretin \cite{MN},
Neretin \cite{Ner-doug}
(we discuss the references in a~more arranged form in
Section~\ref{ss:generalities}).

The purpose of this paper is to give a definition of counterparts
of higher hypergeometric functions $_pF_q(z)$ (formula \eqref{eq:def})
and to obtain an analog of formula \eqref{u},
see Theorem~\ref{th:1}.
The main statements are contained in the present section,
their proofs are done in Section~\ref{section2}.
In Section~\ref{section3} we discuss further properties of our functions.

\subsection{Notation}
Denote by $\Z_+$ the set of integers $\ge 0$, by $\Z_-$ the set of integers $\le 0$.

\subsubsection{Hypergeometric functions.}
Let $r\le s+1$.
Let $a_1,\dots,a_r\in\C$ and $b_1,\dots,b_s\in\C\setminus\Z_-$.
Generalized hypergeometric functions are defined by
\begin{equation}
\hyp{r}{s}\left[\genfrac{}{}{0pt}{}{a_1,\dots,a_r}{b_1,\dots,b_r};z\right]
:=\sum_{m=0}^\infty
\frac{(a_1)_m\cdots (a_r)_m}{m! (b_1)_m\cdots (b_s)_m}\cdot z^n,
\label{eq:rFs}
\end{equation}
where $(a)_m:=a(a+1)\cdots(a+m-1)$ is the Pochhammer symbol.
For $r\le s$ the radius of convergence is $\infty$, for $r=s+1$ it is $1$.
For $r>s+1$ the series diverges.

\subsubsection{Notation for lists}
Let $a_1, \dots, a_p\in \C$ and $h\in\C$.
We use the following notation for lists,
\begin{gather*}
(a):= a_1, \dots, a_p;\\
(a)+h:= a_1+h, \dots, a_p+h;\\
(a)_{\setminus j}:= a_1,\dots,a_{j-1},a_{j+1},\dots,a_p.
\end{gather*}
In particular, we denote \eqref{eq:rFs}
by $\hyp{r}{s}\big[\genfrac{}{}{0pt}{}{(a)}{(b)};z\big]$.

\subsubsection{Double powers}
Denote by $\C^\times$ (resp.\ $\R^\times_+$) the multiplicative group of $\C$
(resp.\ the multiplicative group of positive reals).
By $\T\subset \C^\times$ we denote the subgroup
$|z|=1$, we have $\T\simeq\R/2\pi \Z$,
\begin{equation*}
\C^\times= \T\times \R^\times_+,
\end{equation*}
and $\R^\times_+$ is isomorphic to the additive group $\R$ of real numbers.

The dual group $(\C^\times)^*$, i.e., the group of all homomorphisms
$\C^\times\to \T$, is
\begin{equation*}
(\C^\times)^*= (\T)^*\times (\R^\times_+)^*\simeq \Z\times \R.
\end{equation*}

Denote by $\Lambda_\C$ the set of pairs
$a\big| a'$ such that $a$, $a'\in\C$ and $a-a'\in\Z$.
We write such pairs by
\[
a\big|a'=\tfrac{k+\sigma}2\big|\tfrac{-k+\sigma}2,
\qquad\text{where $k\in\Z$, $\sigma\in\C$}.
\]
By $\Lambda\subset \Lambda_\C$ we denote the subset
consisting of $a|a'$ satisfying $a'+\ov a=0$,
\[
a\big|a'=
\tfrac{k+{\rm i}s}2\big|\tfrac{-k+{\rm i}s}2\in \Lambda,\qquad\text{where $s\in \R$}.
\]

Define the following functions on $\C^\times$:
\[
z^{a|a'}:= z^a \ov z^{\,a'}:= (z/\ov z)^{k/2} |z|^{\sigma}.
\]
Notice that such functions are precisely all homomorphisms
$\C^\times\to \C^\times$.
For $a|a'\in \Lambda$ we get homomorphisms $\C^\times\to \T$,
\[
\Big|z^{\tfrac{k+{\rm i}s}2\big|\tfrac{-k+{\rm i}s}2} \Big|=1.
\]

\subsubsection{The complex Mellin transform}
Denote by $\dd t$ the Lebesgue measure on $\C$,
\[
\dd t:= {\rm d}\Re t\, {\rm d}\Im t.
\]

We define the complex Mellin transform $\cM$ as the Fourier transform on the group $\C^\times$.
For a~function $f$ on $\C^\times$ we define the Mellin transform as a function
on $\Lambda_\C$ defined by
\[
g\big(a\big|a'\big)=
\cM f\big(a\big|a'\big):=\frac 1{2\pi}
\int_{\C^\times} t^{a|a'} f(t)\frac {\dd t}{t^{1|1}},
\]
the factor ${\dd t}/{t^{1|1}}$ is the $\C^\times$-invariant measure on $\C^\times$.

The Mellin transform is a unitary operator from
$L^2\big(\C^\times,{\dd t}/{t^{1|1}} \big)$
to $L^2$ on $\Lambda\simeq \Z\times\R$, the inversion formula is
\begin{equation*}
f(t)=\frac 1{2\pi}\sum_{k\in\Z} \int_\R g\big(\tfrac{k+{\rm i}s}2\big|\tfrac{-k+{\rm i}s}2\big) t^{{\tfrac{-k-{\rm i}s}2\big|\tfrac{k-{\rm i}s}2} }\,{\rm d}s.
\end{equation*}

The convolution on the group
$\C^\times$ is defined by formula
\[ f*g(z):=\int_\C f(t)g(z/t)\,\frac{\dd t}{t^{1|1}}.
\]
The Mellin transform sends convolutions to products,
\[
\cM (f*g)(z)=2\pi\, \cM f(z)\cdot \cM g(z).
\]

\subsection{The Gamma-function of the complex field}\label{ss:Gamma}
Following Gelfand, Graev and Retakh \cite{GGR},
we define the Gamma-function of the complex field as
a function on $\Lambda_\C$ by\footnote{The definition slightly
differs from a definition from \cite{GGR}, which was used in
\cite{MN,Ner-doug}. We write $\Im t$ instead of $\Re t$. For this reason
the factor $i^{a-a'}$ from \cite{GGR,MN} disappears.}
\begin{align}
\Gamma^\C\big(a\big|a'\big)&:=
\frac 1\pi \int_\C t^{a-1|a'-1} {\rm e}^{2{\rm i} \Im t}\dd ta \notag\\
&\hphantom{:}=
\frac{1}{\pi} \lim_{r\to\infty} \int_{|t|\le r} t^{a-1|a'-1}
{\rm e}^{2{\rm i} \Im t}\dd t
= \frac{\Gamma(a)}{\Gamma(1-a')}
=\frac{(-1)^{a-a'}\Gamma(a')}{\Gamma(1-a)} \notag \\
&\hphantom{:}= \frac{1}{\pi}\,
\Gamma(a)\Gamma(a')\sin\pi a'
= \frac{(-1)^{a-a'}}{\pi} \,\Gamma(a)\Gamma(a')\sin\pi a. \label{eq:Gamma}
\end{align}
The integral conditionally converges if $0<\Re(a+a')<1$
and diverges otherwise. Clearly, the right hand side is meromorphic
in the whole $\Lambda_\C$.

\begin{Remark} In particular, we can write \eqref{eq:beta} as
\[
\B^\C(\alpha|\alpha',\beta|\beta')=
\frac{\Gamma^\C(\alpha|\alpha')\,\Gamma^\C(\beta|\beta')}
{\Gamma^\C(\alpha+\beta|\alpha'+\beta')}.
\]
\end{Remark}

It is easy to see that
\begin{subequations}
\begin{gather}
\Gamma^\C\big(a|a'\big)=(-1)^{a-a} \Gamma^\C\big(a'|a\big),\label{eq:symmetry} \\
\Gamma^\C\big(a|a'\big)\, \Gamma^\C\big(1-a|1-a'\big) =(-1)^{a-a'},\label{eq:gamma-complement} \\
\Gamma^\C\big(a+m\big|a'+m'\big) =(-1)^{m'} \Gamma^\C\big(a\big|a'\big)(a)_m(a')_{m'}, \label{eq:shift1} \\
 \Gamma^\C\big(a-m\big|a'-m'\big) =\frac{(-1)^{m} \Gamma^\C\big(a\big|a'\big)}{(1-a)_m(1-a')_{m'}}, \label{eq:shift2} \\
\prod_{j=0}^{k-1}\Gamma^\C\big(a+\tfrac jk\big|a'+\tfrac jk \big)=\Gamma^\C\big(ka\big|ka'\big)k^{1-(a+a')k},\label{eq:gamma-gauss}
\end{gather}
\end{subequations}
where $m\in\N$, $k\in\N$.
Values of $\Gamma^\C$ at integer points are
\begin{gather*}
 \Gamma^\C(k|l) = 0,\qquad \text{for $k$, $l\in\N$;} \\
 \Gamma(m|-k) =\frac{(m-1)!(-1)^k}{k!},\qquad \text{for $m\in \N$, $k\in \Z_+$}.
\end{gather*}
For $m$, $m'\in\Z_+$ we have a pole at $(-m)|(-m')$, more precisely,
\begin{equation}
\res\limits_{\epsilon=0}
\Gamma^\C(-m+\epsilon|-m'+\epsilon)=\frac{(-1)^m}{m!\, m'!}.
\label{eq:res-gamma}
\end{equation}
Next (see Section~\ref{ss:asy} below),
for $a|a'\in\Lambda_\C$ and $\xi=\frac12(k+{\rm i}s)$
with $k\in\Z$, $s\in \R$ we have
\begin{equation}
\Gamma^\C\big(a+\xi|a'-\ov \xi\big)=
\exp\big\{2{\rm i}\Im (\xi\ln \xi-\xi)\big\}\cdot \xi^{a-\tfrac12\big|a'-\tfrac12} \big(1+O\big(|\xi|^{-1}\big)\big).
\label{eq:asy1}
\end{equation}
Therefore,
\begin{equation}
\big| \Gamma^\C\big(a+\xi|a'-\ov \xi\big)\big|\sim |\xi|^{\Re(a+a')-1}.
\label{eq:asy2}
\end{equation}

\subsection{Hypergeometric functions of the complex field}\label{ss:def}
Let $a_1|a'_1,\dots, a_q|a'_q$, $b_1|b'_1,\dots,b_p|b'_p\in \Lambda_\C$.
Temporarily, we assume that they satisfy the conditions:
\begin{equation}
\Re(a_\alpha+a_\alpha')>0,\qquad \Re(b_\beta+b_\beta')>0.
\label{eq:cond1}
\end{equation}
and
\begin{equation}
\sum_\alpha \Re(a_\alpha+a_\alpha')+\sum_\beta\Re(b_\beta+b_\beta')<p+q.
 \label{eq:cond2}
\end{equation}

We define the following function on $\Lambda_\C$:
\begin{gather}
\Kc{p}{q}\left[\genfrac{}{}{0pt}{}{(a|a')}{(b|b')};  \tfrac{k+\sigma}{2}\Bigl| \tfrac{-k+\sigma}{2}
\right]\nonumber\\
\qquad{} :=
 \prod_{\alpha=1}^q \Gamma^\C\Bigl(a_\alpha+\tfrac{k+\sigma}2\Bigl|a_\alpha'+\tfrac{-k+\sigma}2\Bigr)
 \prod_{\beta=1}^p \Gamma^\C\Bigl(b_\beta+\tfrac{-k-\sigma}2\Bigl|b_\beta'+\tfrac{k-\sigma}2\Bigr).
\label{eq:cK}
\end{gather}
Next, we define the hypergeometric functions of the complex field
as the contour integral:
\begin{equation}
\Gc{p}{q}\left[\genfrac{}{}{0pt}{}{(a|a')}{(b|b')};z\right]
:=\frac{1}{2\pi {\rm i}}
\sum_{k\in\Z} \int_{{\rm i}\R}
\Kc{p}{q}\left[\genfrac{}{}{0pt}{}{(a|a')}{(b|b')};\tfrac{k+\sigma}{2}\Bigl| \tfrac{-k+\sigma}{2}\right]
 z^{\frac{-k-\sigma}2\big| \frac{k-\sigma}2} \,{\rm d}\sigma.
 \label{eq:def}
\end{equation}
The integration is taken along the imaginary axis ${\rm i}\R$.
The condition \eqref{eq:cond2} provides the conditional convergence of
the integral in $\sigma$ and the absolute convergence of the series
(see Sections~\ref{ss:meijer} and \ref{ss:evaluation} below).
Under the stronger condition
\begin{equation}
\sum_\alpha \Re(a_\alpha+a_\alpha')+\sum_\beta\Re(b_\beta+b_\beta')<p+q-1
 \label{eq:cond3}
\end{equation}
all integrals in $\sigma$ convergence absolutely
(this follows from \eqref{eq:asy2}).

The expression \eqref{eq:cK} has poles (they are analyzed in
Section~\ref{ss:evaluation}) in $\sigma$ originating from two groups
of factors.
We call poles of the factors
$\Gamma^\C\big(a+\frac{k+\sigma}2\big|a'+\frac{-k+\sigma}2\big)$
\emph{left poles}, the poles of the factors
$\Gamma^\C\big(b+\frac{-k-\sigma}2\big|b'+\frac{k-\sigma}2\big)$
\emph{right poles}.
Under the conditions \eqref{eq:cond1} left poles are contained
in the left half-plane $\Re\sigma<0$,
right poles are contained in the right half plane $\Re\sigma>0$, and
the axis ${\rm i}\R$ separates these groups of poles.

As usual (see \cite{Sla}) we can write the analytic continuation of
$\Gc{p}{q}\big[\genfrac{}{}{0pt}{}{(a|a')}{(b|b')};z\big]$
to a wider domain of $(a|a')$, $(b|b')$ by moving integration contours.
Let us omit the condition~\eqref{eq:cond1} and fix
$a_\alpha|a'_\alpha$, $b_\beta|b_\beta'$.
Then ${\rm i}\R$ separates left and right poles for all but a finite number of
summands.
Assume
\begin{equation}
a_\alpha+b_\beta\notin\Z_+ \qquad\text{for all $\alpha$, $\beta$.}
\label{eq:cond4}
\end{equation}
For each $k$ we choose a contour $L_k$ that coincides with ${\rm i}\R$ at infinity
and separates left and right poles of the corresponding summand,
the result of the integration does not depend on a choice of $L_k$.
Then we replace the expression \eqref{eq:def} by
\[
\sum_k \int_{L_k}
\]
and get an expression of our integral in the domain
defined by conditions \eqref{eq:cond2}, \eqref{eq:cond4}.

\subsection{The statements of the paper}
For the same $(a|a')$, $(b|b')$ we define the expression
$\Sigma^\C_+(z)$ by
\begin{gather}
\Sigma^\C_+
\left[\genfrac{}{}{0pt}{}{(a|a')}{(b|b')};z\right]
:=2\sum_{j=1}^q z^{a_j|a_j'}\cdot \prod_{\beta}
\Gamma^\C\big(b_\beta+a_j\big|b_\beta'+a_j'\big)
\cdot \prod_{\alpha\ne j} \Gamma^\C\big(a_\alpha-a_j\big|a_\alpha'-a_j'\big) \nonumber\\
\qquad{}
\times
\hyp{p}{q-1}\left[\genfrac{}{}{0pt}{}{(b_\alpha+a_j)}
{(1-a_\alpha+a_j)_{\setminus j}};(-1)^q\,z\right]
\hyp{p}{q-1}\left[\genfrac{}{}{0pt}{}{(b_\alpha'+a_j')}
{(1-a_\alpha'+a_j')_{\setminus j}};(-1)^p\,\ov z\right].
 \label{eq:Xi}
\end{gather}
We also define the expression $\Sigma^\C_-(z)$ by
\[
\Sigma^\C_-
\left[\genfrac{}{}{0pt}{}{(a|a')}{(b|b')};z\right]
:=\Sigma^\C_+ \left[\genfrac{}{}{0pt}{}{(b|b')}{(a|a')};z^{-1}\right].
\]

\begin{Theorem}\label{th:1}
Let $(a|a')$, $(b|b')$ satisfy the conditions \eqref{eq:cond1} and~\eqref{eq:cond2}.
\begin{enumerate}\itemsep=0pt
\item[$(a)$] For $q>p$ we have
\[
\Gc{p}{q}\left[\genfrac{}{}{0pt}{}{(a|a')}{(b|b')};z\right]
=\Sigma^\C_+ \left[\genfrac{}{}{0pt}{}{(a|a')}{(b|b')};z\right]
\]
and for $q<p$ we have $_p\G^\C_q(z)=\Sigma^\C_-(z)$.
\item[$(b)$] For $p=q$
\[
\Gc{p}{q}(z)=
\begin{cases}
\Sigma^\C_+(z), & \text{if $|z|<1$}, \\
\Sigma^\C_-(z), & \text{if $|z|>1$}.
\end{cases}
\]
\end{enumerate}
\end{Theorem}

\begin{Remark} This statement and the main argument for its proof
are potentially contained in the paper by Ismagilov \cite[Lemma~2]{Ism2}.
\end{Remark}

\begin{Corollary}\label{cor:differential-eq}
The function $F(z):=\Gc{p}{q}\big[\genfrac{}{}{0pt}{}{(a|a')}{(b|b')};z\big]$
satisfies the following system of differential equations
\begin{equation}
\cD F=0,\qquad \ov\cD F=0,
\label{eq:differential-eq}
\end{equation}
where
\begin{subequations}
\begin{gather}
\cD :=(-1)^{q} \prod_{\alpha=1}^p
\left(z\frac{\partial}{\partial z}+a_\alpha\right) -
z\prod_{\beta=1}^q \left(z\frac{\partial}{\partial z}-b_\beta\right),
\label{eq:cD} \\
\ov\cD :=(-1)^p \prod_{\alpha=1}^p
\left(\ov z\frac{\partial}{\partial \ov z}+a_\alpha'\right)
-\ov z\prod_{\beta=1}^q
\left(\ov z\frac{\partial}{\partial \ov z}-b_\beta'\right).
\label{eq:ovcD}
\end{gather}
\end{subequations}
\end{Corollary}

Consider the expression $\Gc{p}{q}\big[\genfrac{}{}{0pt}{}{(a|a')}{(b|b')};z\big]$
as a function of a variable $z$ and parameters
$(a|a')$, $(b|b')$, i.e., a function defined
on a certain domain in the space{\samepage
\[
\C\times (\Lambda_\C)^{p+q}\simeq \C\times \Z^{p+q}\times \C^{p+q}.
\]
Fixing integer parameters\footnote{The conditions
\eqref{eq:cond1} and \eqref{eq:cond2} allow arbitrary integer parameters.}
$(a-a')$, $(b-b')$ we get a countable collection of functions on domains in $\C\times \C^{p+q}$.

}

Next, fix $\alpha$, $\beta$, fix also $m$, $m'\in\Z_+$.
Consider the subset $S(\alpha, \beta;m,m')$ in $\C\times (\Lambda_\C)^{p+q}$
defined by the equations
\begin{equation}
\begin{cases}
a_\alpha+b_\beta=m, \\
a_\alpha'+b_\beta'=m'.
\end{cases}
\label{eq:collision}
\end{equation}

\begin{Proposition}\label{pr:continuation}
For fixed $(a-a')\in \Z^{p+q}$,
$(b-b')\in\Z^{p+q}$ the expression
$\Gc{p}{q}\big[\genfrac{}{}{0pt}{}{(a|a')}{(b|b')};z\big]$
as a~function of $(a)$, $(b)$, $z$ admits an extension to a~function
which is real analytic in $z$ \textup{(}i.e., real analytic as a function
in two variables $\Re z$, $\Im z$\textup{)} in the domain
\begin{gather*}
z\in \C\setminus\{0\}, \qquad \text{if $p\ne q$}, \\
z\in \C\setminus\{0,(-1)^p\}, \qquad \text{if $p= q$},
\end{gather*}
and meromorphic in $(a)$, $(b)$ with singularities
\textup{(}simple poles\textup{)} located in
\[
\bigcup_{\alpha,\beta,m,m'} S(\alpha, \beta;m,m').
\]
\end{Proposition}

\begin{Lemma}\label{cor:1}\quad
\begin{enumerate}\itemsep=0pt
\item[$(a)$]
Let the parameters $(a|a')$, $(b|b')$ satisfy the conditions
\begin{gather*}
\Re(a_\alpha+a_\alpha')>0,\qquad \Re(b_\beta+b_\beta')>0,\\
\sum_\alpha \Re(a_\alpha+a_\alpha')+\sum_\beta(b_\beta+b_\beta')<p+q-1.
\end{gather*}
Then the function $\Gc{p}{q}\big[\genfrac{}{}{0pt}{}{(a|a')}{(b|b')};z\big]$
is contained in $L^2\big(\C,|z|^{-2}\,\dd z\big)$.
\item[$(b)$] Under the same conditions
\[
 \frac1{2\pi}\int_\C z^{\sigma-1|\sigma'-1}\,
\Gc{p}{q}\left[\genfrac{}{}{0pt}{}{(a|a')}{(b|b')};z\right] \, \dd z=
\Kc{p}{q}\left[\genfrac{}{}{0pt}{}{(a|a')}{(b|b')};\sigma|\sigma'\right],
\]
where the integral is understood as a Mellin transform in $L^2$.
\end{enumerate}
\end{Lemma}

\begin{Theorem}
\label{cor:2}
Let two functions
\[
\Gc{p}{q}\left[\genfrac{}{}{0pt}{}{(a|a')}{(b|b')};z\right],\qquad
\Gc{r}{s}\left[\genfrac{}{}{0pt}{}{(c|c')}{(d|d')};z\right]
\]
be contained in $L^2\big(\C, |z|^{-2}\dd z\big)$.
Then
\begin{equation}
\frac 1{2\pi}
\int_\C
\Gc{p}{q}\left[\genfrac{}{}{0pt}{}{(a|a')}{(b|b')};z\right] \cdot
\Gc{r}{s}\left[\genfrac{}{}{0pt}{}{(c|c')}{(d|d')};\frac{t}{z}\right]
\frac{\dd z}{z^{1|1}} =
\Gc{p+r}{q+s}\left[\genfrac{}{}{0pt}{}{(a|a'),(c|c')}
{(b|b'),(d|d')};t\right].
\label{eq:convolutions}
\end{equation}
\end{Theorem}

Taking $p=q=1$ (see \eqref{eq:1G1}) we get a counterpart of formula
\eqref{eq:inductive}:
\begin{gather*}
\frac 1{2\pi}\int_\C z^{a|a'}(1+z)^{-a-b|-a'-b'} \cdot
\Gc{r}{s}\left[\genfrac{}{}{0pt}{}{(c|c')}{(d|d')};\frac{t}{z}\right]
\frac{\dd z}{z^{1|1}} \\
\qquad{} =\frac 1{2\Gamma^\C(a+b|a'+b')}\,
\Gc{1+r}{1+s}\left[\genfrac{}{}{0pt}{}{a|a',(c|c')}{b|b',(d|d')};t\right].
\end{gather*}

\section{Proofs}\label{section2}

\subsection[Asymptotics of $\Gamma^\C$]{Asymptotics of $\boldsymbol{\Gamma^\C}$}\label{ss:asy}
Let us derive formula \eqref{eq:asy1} for the asymptotics of the function $\Gamma^C(a+\xi|a'-\ov \xi)$.
First, let us verify that the right-hand side is single-valued
on $\Lambda$.
We must verify that the expression
\[
\exp\big\{ \xi\ln\xi - \ov\xi\ln\ov\xi\big\}, \qquad \xi=\tfrac{k+{\rm i}s}2
\]
is single valued.
Represent $\xi=r {\rm e}^{{\rm i}\phi}$. We transform our expression as
\begin{gather*}
\exp\big\{\xi(\ln(r)+{\rm i}\phi+2\pi {\rm i} n)-
\ov \xi(\ln(r)-{\rm i}\phi-2\pi {\rm i} n) \big\} \\
\qquad{}=\exp\big\{(\xi-\ov\xi) \ln r +(\xi+\ov\xi)({\rm i}\phi+2\pi {\rm i} n) \big\}.
\end{gather*}
We have $\xi+\ov\xi\in \Z$ and therefore the exponential is single valued.

Next, we apply the Stirling formula in the form
(see \cite[equation~(1.18(3))]{HTF1})
\[
\ln
\Gamma(c+z)=\big(z+c-\tfrac12\big)\ln z -z+\tfrac 12\ln(2\pi)+O\big(z^{-1}\big),
\qquad \text{where $|\arg (z)|<\pi-\epsilon$},
\]
and get
\begin{gather*}
\Gamma(a+\xi) \simeq \sqrt{2\pi}\exp
\big\{ \big(\xi+a-\tfrac12\big)\ln\xi-\xi \big\},\\
\Gamma(1-a'-\ov \xi) \simeq \sqrt{2\pi}\exp
\big\{ \big(\ov\xi-a'+\tfrac12\big)\ln\ov\xi-\ov\xi\big\}.
\end{gather*}
The ratio gives us \eqref{eq:asy1}.
An evaluation of the asymptotics in the sector $|\arg(-z)|<\pi-\epsilon$
gives the same result.

\begin{Corollary}\label{cor:L2}
Let $\Re(a_\alpha+a_\alpha')>0$, $\Re(b_\beta+b_\beta')>0$.
Denote
\[
\upsilon:=\sum (a_\alpha+a_\alpha')+\sum(b_\beta+b_\beta').
\]
\begin{enumerate}\itemsep=0pt
\item[$(a)$] If $\upsilon<p+q$, then $F(k,{\rm i}s):=
\Kc{p}{q}\big[\genfrac{}{}{0pt}{}{(a|a')}{(b|b')};\tfrac{k+{\rm i}s}{2}\bigl|;\tfrac{-k+{\rm i}s}{2}\big]$
tends to $0$ as $|k+{\rm i}s|$ tends to $\infty$.
\item[$(b)$]
If $\upsilon<p+q-1$, then for each $k$ the function $F(k,{\rm i}s)$ as a
function in $s$ is integrable.
Under the same condition $F(k,{\rm i}s)$ is contained in $L^2(\Lambda)$.
\item[$(c)$]
If $\upsilon<p+q-2$, then $F(k,{\rm i}s)$ is contained in $L^1(\Lambda)$.
\end{enumerate}
\end{Corollary}

\begin{Remark}
Formula \eqref{eq:asy1} also gives us asymptotics of
$\Kc{p}{q}$ on vertical lines $\sigma=h+{\rm i}s$. Indeed,
\[
\Gamma\big(a+\tfrac{k+h+{\rm i}s}{2}\big|a'+\tfrac{-k+h+{\rm i}s}{2}\big)=
\Gamma\big(a+\tfrac h2+\tfrac{k+{\rm i}s}{2}\big|a'+\tfrac h2+\tfrac{-k+{\rm i}s}{2}\big),
\]
and we can control integrability under shifts of the integration contour.
\end{Remark}

\subsection{Decomposition of Mellin--Barnes integrals in residues}\label{ss:meijer}
Now we start to prove Theorem~\ref{th:1}, i.e., to derive
the quadratic expressions of the func\-tions~$\Gc{p}{q}$ in terms of
the usual hypergeometric functions $\hyp{p}{q}$.

First, we write the definition \eqref{eq:def} of $\Gc{p}{q}$ in the form
\[
\sum_{k\in \Z} z^{-k/2\big|k/2} I_k(z), \qquad
I_k(z)=\frac{1}{2\pi {\rm i}}\int (\dots)|z|^{-\sigma}\,{\rm d}\sigma.
\]

The integrals $I_k(z)$ are special cases of
Mellin--Barnes integrals, i.e., integrals of the type
\begin{equation}
J^{A,B}_{C,D}(a,b,c,d;u)=\frac 1{2\pi {\rm i}}\int_L \frac{\prod\limits_{\alpha=1}^A\Gamma(a_\alpha+\sigma)
\prod\limits_{\beta=1}^B\Gamma(b_\beta-\sigma)}
{\prod\limits_{\gamma=1}^C\Gamma(c_\gamma+\sigma)\prod\limits_{\delta=1}^D\Gamma(d_\delta-\sigma)} u^{-s}
\, {\rm d}s,
\label{eq:J}
\end{equation}
where $L$ is a contour separating poles of factors
$\Gamma(a_\alpha+\sigma)$ and poles of $\Gamma(b_\beta-\sigma)$.
The behavior of such Mellin--Barnes integrals
(Meijer $G$-functions) was investigated by Meijer,
his results are exposed in \cite{BW, Luke,Mar}.
Under certain conditions $J(a,b,c,d;u)$ can be expressed in terms
of functions $\Sigma_\pm(z)$, where
$\Sigma_+(z)$ is the sum of residues of the integrand at poles of
factors $\Gamma(a_\alpha+\sigma)$, and $\Sigma_-(z)$
the sum of residues at poles of $\Gamma(b_\beta-\sigma)$.

In our case $A=D=q$, $B=C=p$, and $u=|z|$ is a positive real.
The contour $L$ coincides with the imaginary axis at infinity.
The asymptotics of the absolute value of the integrand on the imaginary axis
is
\[
\sim |\sigma|^{\sum\limits_{\alpha=1}^q\Re (a_\alpha-d_\alpha)+\sum\limits_{\beta=1}^p\Re(b_\beta- c_\beta)},
\]
see \cite[equation~(1.18(4))]{HTF1}.
One of the statements of \cite[Theorem~18]{Mar} implies that
\emph{if the integrand tends to zero on the imaginary axis,
then the integral $J(a,b,c,d;u)$ conditionally converges if
the integrand tends to $0$ at infinity for $u>0$ and is given by}
\begin{itemize}\itemsep=0pt
\item[--] \textit{$\Sigma_+(u)$ for $q>p$;}
\item[--] \textit{$\Sigma_-(u)$ for $q<p$;}
\item[--] \textit{for $p=q$ we have $\Sigma_+(u)$ if $0<u<1$ and $\Sigma_-(u)$ for $u>1$.}
\end{itemize}
If
$\sum\Re(a_\alpha-d_\alpha)+\sum\Re(b_\beta-c_\beta)<-1$,
then the absolute convergence is obvious.

In our case $a_\alpha$ is replaced by $a_\alpha+k/2$, $b_\beta$ by
$b_\beta-k/2$, and $c_\beta\mapsto 1-b'_\beta-k/2$,
$d_\alpha\mapsto 1-a'_\alpha+k/2$.
Therefore under our conditions \eqref{eq:cond1}, \eqref{eq:cond2}
all integrals $I_k$ converge.

\subsection{Evaluation of sums of residues}\label{ss:evaluation}

\begin{Lemma}
Fix $a|a'$ and consider the following family of functions:
\[
\gamma_k(\sigma):=\Gamma^\C\big(a+\tfrac {k+\sigma}2\big|a'+\tfrac
{-k+\sigma}2\big).
\]
Denote by $\Omega$ the sets of all points $(k,\sigma)\in \Z\times \C$
such that $\sigma$ is a pole of $\gamma_k(\sigma)$.
\begin{enumerate}\itemsep=0pt
\item[$(a)$] The set $\Omega$ is contained in the half-plane $\Re\sigma<0$ if and only
if $\Re(a+a')>0$.
\item[$(b)$] The set $\Omega$ consists of points
\begin{gather*}
k=-m+m'-a+a', \\
\sigma= -m-m'-a-a',
\end{gather*}
where $m$, $m'$ range in $\Z_+$.
\end{enumerate}
\end{Lemma}

\begin{proof} (a)
Let all poles be contained in the domain $\Re\sigma<0$.
Taking $k=a'-a$ we get
\[
\gamma_{a'-a}(\sigma)=\Gamma^\C
\big(a+\tfrac {a'-a+\sigma}2\big|a'+\tfrac {-a'+a+\sigma}2\big)=
\Gamma^\C\big(\tfrac{a+a'}2 +\tfrac{\sigma}{2}\big|\tfrac{a+a'}2
+\tfrac{\sigma}{2}\big).
\]
The point $\sigma= -(a+a')$ is a pole of $\gamma_{a'-a}(\sigma)$.
Therefore, $\Re(a+a')>0$.

Conversely, let $\Re(a+a')>0$.
Consider $(k,\sigma)\in \Omega$.
Then (see Section~\ref{ss:Gamma})
\begin{equation}
a+\tfrac{k+\sigma}2=-m\in \Z_-, \qquad a'+\tfrac{-k+\sigma}2=-m'\in \Z_-.
\label{eq:Omega}
\end{equation}
Therefore $a+a'+\sigma\in \Z_-$ and $\Im\sigma<0$.

(b) We solve the system \eqref{eq:Omega}.
\end{proof}

\begin{proof}[Evaluation of the sum of residues (the proof of Theorem~\ref{th:1})]
For definiteness assume that $p\le q$.
Let us write the residue $R(j;m,m')$ of the integrand \eqref{eq:def}
at a point
\[
k= -m+m'-a_j+a'_j,\qquad \sigma=-m-m'-a_j-a'_j.
\]
We have
\[
\tfrac{k+\sigma}{2}=-a_j-m,\qquad \tfrac{-k+\sigma}{2}=-a_j'-m',
\]
and
keeping in mind \eqref{eq:res-gamma} we get
\begin{gather*}
R(j;m,m')=\frac{2(-1)^m}{m!\,m'!}
 \prod_{\alpha\ne j}
\Gamma^\C\big(a_\alpha-a_j-m\big|a_\alpha'-a_j'-m'\big)\\
\phantom{R(j;m,m')=}{} \times
\prod_{\beta} \Gamma^\C\big(b_\beta+a_j+m\big|b_\beta'+a_j'+m'\big)
\times z^{a_j+m|a_j'+m'}.
\end{gather*}
Applying \eqref{eq:shift1} and \eqref{eq:shift2} we come to
\begin{gather*}
R(j;m,m')=2 \frac{(-1)^m}{m!\,m'!}
\prod_{\alpha\ne j} \frac{\Gamma^\C\big(
a_\alpha-a_j\big|a_\alpha'-a_j'\big)(-1)^{m}}
{(1-a_\alpha+a_j)_m(1-a_\alpha'+a_j')_{m'}} \\
\hphantom{R(j;m,m')=}{} \times
\prod_{\beta} \Gamma^\C\big(b_\beta+a_j\big|b_\beta'+a_j'\big)
b_\beta+a_j\big)_m(b_\beta'+a_j')_{m'} (-1)^{m'}
\times z^{a_j+m|a_j'+m'}.
\end{gather*}
Reordering factors, we get
\begin{gather*}
2 z^{a_j|a'_j}
\prod_{\alpha\ne j} \Gamma^\C\big(a_\alpha-a_j\big|a_\alpha'-a_j'\big)
\prod_{\beta}
\Gamma^\C\big(b_\beta+a_j\big|b_\beta'+a_j'\big) \\
\qquad{}\times
(-1)^{mq} \frac{z^m \prod\limits_{\beta}(b_\beta+a_j\big)_{m}}
{m!\prod\limits_{\alpha\ne j} (1-a_\alpha+a_j)_{m}}
 \times
(-1)^{m'p} \frac{\ov z^{m'} \prod\limits_{\beta}(b_\beta'+a_j'\big)_{m'}}
{m'!\prod\limits_{\alpha\ne j} (1-a_\alpha'+a_j')_{m'}}.
\end{gather*}

Formal calculation with series gives
\begin{gather*}
\sum_{m,m'} R(j,m,m')=
2 z^{a_j|a_j'}
 \prod_{\alpha\ne j} \Gamma^\C\big(a_\alpha-a_j\big|a_\alpha'-a_j'\big)
 \prod_{\beta}
 \Gamma^\C\big(b_\beta+a_j\big|b_\beta'+a_j'\big)
\\
\hphantom{\sum_{m,m'} R(j,m,m')=}{} \times
 \sum_m \frac{ \big((-1)^{q} z\big)^m \prod\limits_{\beta}(b_\beta+a_j\big)_m} {m!\prod\limits_{\alpha\ne j} (1-a_\alpha+a_j)_m}
 \times
 \sum_{m'} \frac{ \big((-1)^{p} \ov z\big)^{m'} \prod\limits_{\beta}(b_\beta+a_j\big)_{m'}}
 {m'!\prod\limits_{\alpha\ne j} (1-a_\alpha'+a_j')_{m'}} .
 \end{gather*}
We get a product of two hypergeometric series whose radius of convergence
is $\infty$ if $q>p$ and $1$ if $q=p$.
They are absolutely convergent in the disk of convergence,
therefore the last identity really takes place.
Therefore
\[
\sum_j \sum_{m,m'} R(j,m,m')=\Sigma^\C_+ (z).
\]
On the other hand the absolute convergence allows us to write
 \begin{equation}
 \sum_{m,m'} R(j,m,m')=\sum_{k\in\Z}\,\,
\sum_{m,m'\colon m-m'=k} R(j,m,m')=\sum_{k\in\Z} (z/\ov z)^{k/2} I_k.
 \label{eq:sum-R}
 \end{equation}
Thus we proved the coincidence of $\Gc{p}{q}$ and $\Sigma_+^\C$.
The summation of residues at the right poles of
$_p^{\vphantom{\C}}\cK^\C_q$ is similar.

It is important that the convergence of the series $\sum\limits_{k\in\Z}$
in~\eqref{eq:sum-R} is locally uniform in the parameters
$a_\alpha$, $b_\beta$ near any point $(a_0)\in\C^q$, $(b_0)\in\C^p$,
for which the coefficients at $z^l\ov z^{l'}$ are well-defined.
\end{proof}

\subsection{The application of the Mellin transform}
Our next purpose is Theorem~\ref{cor:2} about evaluations
of the convolution integrals
\[
\int_{\C} \Gc{p}{q}(\dots;z)\, \Gc{r}{s}(\dots; t/z) \,\frac{\dd z}{z^{1|1}}.
\]
Under the condition of Corollary~\ref{cor:L2} the integral \eqref{eq:def}
defining a $\Gc{p}{q}$-function is the inverse Mellin transform
 of the function $\Kc{p}{q}$.
Therefore the function $\Kc{p}{q}$ is the direct Mellin transform
of $\Gc{p}{q}$.
This is the statement of Lemma~\ref{cor:1}.

The Mellin transform is the Fourier transform
on the group $\C^\times$, therefore it sends multiplicative convolutions
to products. This property remains to be valid if both functions are contained
 in $L^2$, see \cite[Theorem 64, Section~3.13]{Tit}.\footnote{{\bf Proof.} Denote by $C_0$ the space of continuous functions that have
 zero limit at infinity. The product is a~continuous operation $L^2\times L^2\to L^1$. The Fourier transform
 sends $L^2\to L^2$ bijectively, and sends $L^1\to C_0$.
 The convolution is a continuous operation $L^2\times L^2\to C_0$.
 Therefore the Fourier transform sends a product of two $L^2$-functions
 to the convolution.} This implies Theorem~\ref{cor:2}.

\subsection{The differential equations}
Here we derive the system \eqref{eq:differential-eq}
of differential equations for a function $\Gc{p}{q}$
(the statement of Corollary~\ref{cor:differential-eq}).
We can use the explicit expressions $\Sigma_\pm^\C$ for
$\Gc{p}{q}$ obtained in Theorem~\ref{th:1}.
The formula \eqref{eq:Xi} for $\Sigma_+^\C$ has the form
\[
\sum C_j \Phi_j(z)\cdot \Psi_j(\ov z),
\]
where $C_j$ do not depend on $z$ and
\[
\Phi_j(z)= z^{a_j} \hyp{p}{q-1}
\left[\genfrac{}{}{0pt}{}{(b_\alpha+a_j)}{(1-a_\alpha+a_j)_{\setminus j}};
(-1)^{p+q} z\right].
\]
The equation $\cD F=0$ of the system \eqref{eq:differential-eq}
is slightly modified equation for generalized hypergeometric functions,
see \cite[Section~2.1.2]{Sla} or \cite[Section~5.2]{Luke}.
The functions $\Phi_j(z)$ form a fundamental system of solutions
of this equation in holomorphic functions on~$\C$.

This implies Corollary~\ref{cor:differential-eq}.

\begin{Remark}
The statement also follows from the
identities \eqref{eq:A-diff1}--\eqref{eq:A-diff4} established below.
\end{Remark}

\subsection{The analytic continuation}
It remains to prove Proposition~\ref{pr:continuation} about the analytic
continuation of the function
$\Gc{p}{q}\big[\genfrac{}{}{0pt}{}{(a|a')}{(b|b')};z\big]$
in the variable $z$ and in the parameters $(a|a')$, $(b|b')$.

\textbf{(1).}
Assume $p\ne q$, for definiteness let $p<q$.
First, we omit the condition \eqref{eq:cond1}.
A value of $\Gc{p}{q}$ at $z$ is defined by an infinite sum of contour
integrals~$I_k$ and the convergence is locally
uniform in the parameters.
For all but a finite number of summands the imaginary axis ${\rm i}\R$
separates the left poles and right poles of the integrand
(see the terminology of Section~\ref{ss:def}).
Consider one of the remaining summands, say $I_m$.
A contour~$L_m$ separating left poles and right poles
of~\eqref{eq:def} exists if and only if left poles differ from right poles.
The condition \eqref{eq:collision} is the condition of collision of
left and right poles.
Two separation contours $L_m$, $L_m'$ can be nonhomotopic.
However
\[
\frac{1}{2\pi {\rm i}}\left(\int_{L_m}-\int_{{\rm i}\R}\right)
\]
is a sum of residues at right poles contained in the half-plane
$\Re\sigma<0$ minus a sum of residues at left poles contained in
the half-plane $\Re\sigma>0$.
This expression does not depend on a choice of separation contour.
Clearly, a contour integral is holomorphic in the parameters.

So the sum of contour integrals is well-defined in the domain
defined by conditions \eqref{eq:cond3}, \eqref{eq:cond4}
(absolute convergence and absence of collisions).

On the other hand the generalized hypergeometric functions
$\hyp{p}{q-1}$ meromorphically depend on their parameters,
and therefore $\Sigma_+(z)$ meromorphically depends
on parameters and this allows to omit the condition \eqref{eq:cond2}
for the convergence. Sums of power series are real analytic and this implies
real analyticity of $\Sigma_+(z)$ for $z\in\C\setminus 0$.

Possible poles are located at hyperplanes
\[
\begin{cases}
a_j+b_\beta\in \Z_-,\\
a_j'+b_\beta'\in \Z_-,
\end{cases}
\]
for some $j$ and $\beta$ and
\[
\begin{cases}
a_\alpha-a_j\in \Z, \\
a_\alpha'-a_j'\in \Z,
\end{cases}
\]
for some $j$, $\alpha$.
It remains to notice that singularities of the second type are removable
under the open condition \eqref{eq:cond3} and therefore they are always
removable.

\textbf{(2).} Let $p=q$. The same argument can be applied
for $|z|< 1$ and $|z|>1$, it remains to examine the function
on the circle $|z|=1$.
We extend $\Gc{p}{q}(z)$ to the same domain of parameters as above
and come away to a smaller domain\footnote{So we assume $p>2$.
For the case $p=2$ the analyticity on the circle is clear
from explicit formulas, see \cite{MN}; the case
$p=1$ is trivial, see \eqref{eq:1G1} below.}
\begin{equation}
\sum\Re(a_\alpha+a_\alpha')+\sum\Re(b_\beta+b_\beta')< p-2.
\label{eq:eshcho}
\end{equation}
Then the contour integral in \eqref{eq:def} has continuous partial
derivatives up to order $p$.
By continuity the equation $\cD F=0$
is valid on the circle $|z|=1$.
The operator $\cD$ is elliptic for $z\ne 0$, $\pm 1$.
Therefore under the condition
\eqref{eq:eshcho} solutions of the equation $\cD F=0$ are analytic
(see, e.g., \cite[Theorem 8.5.1]{Her}).
Consider a point $z_0\ne 1$ on the circle $|z|=1$.
Any generalized hypergeometric function
defined in the disk $|z|<1$ has an analytic continuation to a
neighborhood $U$ of $z_0$.
Therefore the expression \eqref{eq:Xi} for $\Gc{p}{q}(z)$ provides us
an analytic continuation
of $\Sigma_+(z)$ to $U$ (and coincides with $\Sigma_-^\C$ in the intersection
of $U$ and the domain $|z|>1$). This expression is meromorphic in $(a)$, $(b)$ as above.

\section{Additions}\label{section3}

\subsection{Some simple cases}

\subsubsection{The exponential}
For $p=0$, $q=1$ formula (\ref{eq:Xi}) gives
 \begin{equation}
\Gc{0}{1}\left[\genfrac{}{}{0pt}{}{a|a'}{\text{--}};z\right]
=2 z^{a|a'} {\rm e}^{-z} {\rm e}^{\ov z}.
\label{eq:1G0}
\end{equation}

\subsubsection{The power function}
For $p=q=1$ we get
\begin{align}
\Gc{1}{1}\left[\genfrac{}{}{0pt}{}{a|a'}{b|b'};z\right]
&=2 z^{a|a'} \Gamma^\C(a+b|a'+b')
\hyp{1}{0}\left[\genfrac{}{}{0pt}{}{a+b}{\text{--}};-z\right]
\cdot
\hyp{1}{0}\left[\genfrac{}{}{0pt}{}{a'+b'}{\text{--}};-\ov z\right] \notag\\
&=2\Gamma^\C(a+b|a'+b')\, z^{a|a'} (1+z)^{-a-b|-a'-b'}.\label{eq:1G1}
\end{align}
The series $\hyp{1}{0}$ converge in the disc $|z|<1$, but the final
expression is well defined in $\C\setminus\{0,-1\}$.

\subsubsection{The Kummer functions}
For $p=1$, $q=2$ we get
\begin{gather}
\Gc{1}{2}\left[\genfrac{}{}{0pt}{}{a_1|a_1',a_2|a_2'}{b|b'};z\right]
=2 \Gamma^\C(a_1+b|a_1'+b)\,\Gamma^\C(a_2-a_1|a_2'-a_1')\nonumber\\
\hphantom{\Gc{1}{2}\left[\genfrac{}{}{0pt}{}{a_1|a_1',a_2|a_2'}{b|b'};z\right]=}{}
\times z^{a_1}
\hyp{1}{1}\left[\genfrac{}{}{0pt}{}{a_1+b}{1-a_2+a_1};z\right]
 \ov z^{a_1'}
\hyp{1}{1}\left[\genfrac{}{}{0pt}{}{a_1'+b'}{1-a_2'+a_1'};-\ov z\right]\label{eq:F11} \\
\hphantom{\Gc{1}{2}\left[\genfrac{}{}{0pt}{}{a_1|a_1',a_2|a_2'}{b|b'};z\right]=}{}
+ \big\{\text{similar term obtained by the transposition $a_1|a_1'\longleftrightarrow a_2|a_2'$}\big\}.\nonumber
\end{gather}
Denote the hypergeometric functions in these expression by
$\Phi_1(z)$, $\Phi_1'(\ov z)$, $\Phi_2(z)$, $\Phi_2'(\ov z)$.
Then $z^{a_1}\Phi_1(z)$,
$z^{a_2}\Phi_2(z)$ is a fundamental system of holomorphic solutions of
the equation $\cD F=0$ of the system~\eqref{eq:differential-eq},
and $\ov z^{a_1'}\Phi_1'(\ov z)$,
$\ov z^{a_2'}\Phi_2'(\ov z)$ is a fundamental system of antiholomorphic
solutions of the equation $\ov\cD F=0$
(generally, all these functions are ramified at $0$ and $\infty$).
Then the functions
\[
z^{a_1}
\Phi_1(z) \,\ov z^{a_1'}\Phi_1'(\ov z),\quad
z^{a_2}\Phi_2(z)\, \ov z^{a_1'}\Phi_1'(\ov z),
\quad z^{a_1}\Phi_1(z)\, \ov z^{a_2'}\Phi_2'(\ov z),\quad
z^{a_2}\Phi_2(z)\,\ov z^{a_2'}\Phi_2'(\ov z)
\]
is a basis of the space of solutions of the system \eqref{eq:differential-eq}
in a neighborhood of any point $z_0\ne 0$, see \cite[Proposition~3.8]{MN}.
For $a_1-a_2\notin \Z$ solutions non-ramified at $0$ and $\infty$
have the form
\begin{equation}
C_1\, z^{a_1|a_1'}
\Phi_1(z)\Phi_1'(\ov z)+C_2\,z^{a_2|a_2'}
\Phi_2(z)\Phi_2'(\ov z).
\label{eq:lin-comb}
\end{equation}
The asymptotic expansion of confluent hypergeometric function
$\hyp{1}{1}(z)$ as $z\to\infty$ (see, e.g., \cite[Section 4.7]{Luke})
in the sector $|\arg z|<\pi-\epsilon$ is
\begin{gather*}
\hyp{1}{1}\left[\genfrac{}{}{0pt}{}{a}{b};z\right]
 =\frac{\Gamma(b)}{\Gamma(b-a)} {\rm e}^{{\rm i}\pi \sgn (\Im a)} z^{-a}
\left(\sum_{n=0}^{K-1} \frac{(a)_n(1+a-b)_n}{n!} (-z)^{-n} +O\big(z^{-K}\big)
\right) \\
\hphantom{\hyp{1}{1}\left[\genfrac{}{}{0pt}{}{a}{b};z\right]=}{}
+\frac{\Gamma(b)}{\Gamma(a)} e^z z^{a-b}\left(
\sum_{n=0}^{L-1} \frac{(b-a)_n(1-a)_n}{n!} z^{-n}+O\big(z^{-L}\big)\right).
\end{gather*}
For generic $C_1$, $C_2$ the growth of expression \eqref{eq:lin-comb}
as $|\Re z|\to \infty$ is exponential.
For the linear combination \eqref{eq:F11} all exponential terms of the
asymptotics disappear (and ratio of coefficients in~\eqref{eq:F11}
is uniquely defined by this condition, existence of such ratio a priori
is non-obvious).
A reminder is $O\big(z^{-N} {\rm e}^{z}\big)+O\big(z^{-N} {\rm e}^{-z}\big) +O(1)$ for any $N$,
is too rough.
However, under the conditions $\Re(a_1+a_1')>0$, $\Re(a_2+a_2')>0$,
$\Re(a_1+a_1'+a_2+a_2')<1$ our function $\Gc{2}{1}$ is contained in
$L^2\big(\C,|z|^{-2}\dd z\big)$.

\subsubsection{The Bessel functions}
For $p=0$, $q=2$ we get an expression
\begin{gather*}
 \Gc{0}{2}\!\left[\genfrac{}{}{0pt}{}{a_1|a_1',a_2|a_2'}{\text{--}};z\right]
 =2 \Gamma^\C(a_2-a_1|a_2'-a_1') z^{a_1|a_1'}
\,
\hyp{0}{1}\!\left[\genfrac{}{}{0pt}{}{\text{--}}{1+a_1-a_2};z\right]
\hyp{0}{1}\!\left[\genfrac{}{}{0pt}{}{\text{--}}{1+a_1'-a_2'};\ov z\right] \\
\hphantom{\Gc{0}{2}\!\left[\genfrac{}{}{0pt}{}{a_1|a_1',a_2|a_2'}{\text{--}};z\right]=}{}
+\big\{\text{similar term obtained by the transposition
$a_1|a_1'\longleftrightarrow a_2|a_2'$}\big\}.
\end{gather*}

\subsubsection{The Gauss hypergeometric functions}
For $p=q=2$
\begin{gather*}
\Gc{2}{2}\left[\genfrac{}{}{0pt}{}{a_1|a_1',a_2|a_2'}
{b_1|b_1',b_1|b_1};z\right]
=z^{a_1|a_1'} \Gamma^\C(a_2-a_1|a_2'-a_1')
\prod_{\beta=1}^2\Gamma^\C(b_\beta+a_1|b_\beta'+a_1') \\
\hphantom{\Gc{2}{2}\left[\genfrac{}{}{0pt}{}{a_1|a_1',a_2|a_2'}
{b_1|b_1',b_1|b_1};z\right]=}{} \times
\hyp{2}{1}\left[\genfrac{}{}{0pt}{}{b_1+a_1,b_2+a_1}{1+a_1-a_2};z\right] \,
\hyp{2}{1}\left[\genfrac{}{}{0pt}{}{b_1'+a_1',b_2'+a_1'}{1+a_1'-a_2'};z\right] \\
\hphantom{\Gc{2}{2}\left[\genfrac{}{}{0pt}{}{a_1|a_1',a_2|a_2'}
{b_1|b_1',b_1|b_1};z\right]=}{}
+\big\{\text{similar term obtained by the transposition
$a_1|a_1'\longleftrightarrow a_2|a_2'$}\big\}.
\end{gather*}
On the other hand Gelfand, Graev and Retakh \cite{GGR} defined
the analog of the Gauss hypergeometric function by the Euler integral
(see the detailed examination in \cite[Section~3]{MN}):
\begin{gather}
\hypc{2}{1}\left[\genfrac{}{}{0pt}{}{A|A', B|B'}{C|C'};z\right]:=
\frac{\Gamma^{\C}(B|B')}{\Gamma^{\C}(B|B')\,\Gamma^{\C}(C-B|C'-B')} \nonumber\\
\hphantom{\hypc{2}{1}\left[\genfrac{}{}{0pt}{}{A|A', B|B'}{C|C'};z\right]:=}{} \times \int_\C t^{B-1|B'-1}(1-t)^{C-B-1|C'-B'-1}(1-zt)^{-A|-A'} \dd t.
\label{eq:Euler}
\end{gather}
We have
\[
\hypc{2}{1}\left[\genfrac{}{}{0pt}{}{A|A', B|B'}{C|C'};z\right]
=\frac{\Gamma(C|C')(-1)^{C-C'}}{\Gamma(A|A')\,\Gamma(B|B')}\,\,
\hyp{2}{2}\left[\genfrac{}{}{0pt}{}{0|0, 1-C|1-C'}{A|A', B|B' };z\right].
\]
We do not know which notation is better. In any case, for the notation
$\hypc{2}{1}$ formulas are precisely parallel to the classical
theory of the Gauss hypergeometric functions.

As for the Kummer and Bessel cases the system \eqref{eq:differential-eq}
has $4$-dimensional space of local solutions
and a two-dimensional
subspace of solutions that are non-ramified at $0$.
The function $\hypc{2}{1}$ is selected from this subspace by the condition
of non-ramification at $z=1$, see~\cite[Proposition~3.11]{MN}.

\subsection[Some simple properties of the functions $\Gc{p}{q}$]{Some simple properties of the functions $\boldsymbol{\Gc{p}{q}}$}\label{ss:simple}

Here we imitate properties of the Meijer $G$-function
(see Prudnikov, Brychkov and Marichev \cite[Vol.~3, Section 8.2]{PBM}).
Clearly, the function
$\Gc{p}{q}$ is symmetric with respect to permutations of
$a_1|a_1',\dots,a_p|a_p'$ and of $b_1|b_1',\dots,b_q|b_q'$.
If $a_m-a_l\in \Z$, then we have an additional symmetry
\begin{equation*}
\Gc{p}{q}\left[\genfrac{}{}{0pt}{}{(a|a')}{(b|b')};z\right]
=\Gc{p}{q}\left[\genfrac{}{}{0pt}{}{(a|a')_{\setminus m,l},a_m|a_l',a_m'|a_l}
{(b|b')};z\right],\qquad\text{if $a_k-a_l\in\Z$}
\end{equation*}
(if $a_k-a_l\notin\Z$, then the right hand side makes no sense),
this follows from \eqref{eq:Gamma}.

Changing the variables $k\mapsto -k$ and $\sigma\mapsto-\sigma$
in \eqref{eq:def} we get
\begin{equation}
\Gc{p}{q}\left[\genfrac{}{}{0pt}{}{(a|a')}{(b|b')};z\right]
=\Gc{q}{p}\left[\genfrac{}{}{0pt}{}{(b|b')}{(a|a')};z^{-1}\right].
\label{eq:symmetry1}
\end{equation}
Changing only the summation index $k\mapsto -k$, we get
\[
\Gc{p}{q}\left[\genfrac{}{}{0pt}{}{(a|a')}{(b|b')};z\right]
=(-1)^{\sum(a_\alpha-a_\alpha')+\sum(b_\beta-b_\beta')} \,
\Gc{p}{q}\left[\genfrac{}{}{0pt}{}{(a|a')}{(b|b')};\ov z\right].
\]
Keeping in the mind the reflection formula (\ref{eq:gamma-complement}), we get
\begin{equation}
\Gc{p+1}{q+1}\left[\genfrac{}{}{0pt}{}{(a|a'),}{(b|b'),}
\genfrac{}{}{0pt}{}{c|c'}{1-c|1-c'};z\right]
=(-1)^{c-c'}\,
\Gc{p}{q}\left[\genfrac{}{}{0pt}{}{(a|a')}{(b|b')};z\right].
\label{eq:A-cancelation}
\end{equation}
Shifting variables $k\mapsto k+l$, $\sigma\mapsto\sigma+\tau$,
we come to
\begin{equation}
z^{c|c'}
\Gc{p}{q}\left[\genfrac{}{}{0pt}{}{(a|a')}{(b|b')};z\right]
=\Gc{p}{q}\left[\genfrac{}{}{0pt}{}{(a|a')+c|c'}{(b|b')-c|c'};z\right].
\label{eq:shift}
\end{equation}

Keeping in mind \eqref{eq:gamma-complement}, we obtain
\begin{gather}
\Gc{mp}{mq}\left[\genfrac{}{}{0pt}{}
{(a|a'),(a|a')+\frac1m|\frac1m,\dots,(a|a')+\frac{m-1}m|\frac{m-1}m}
{(b|b'),(b|b')+\frac1m|\frac{1^{\vphantom{R}}}m,\dots,(b|b')+\frac{m-1}m|\frac{m-1}m};z\right] \nonumber\\
\qquad{} =m^{p+q-2-\sum(a_\alpha+a_\alpha')-\sum(b_\beta+b_\beta')}
\sum_{l=0}^{m-1}
\Gc{p}{q}\left[\genfrac{}{}{0pt}{}{(ma|ma')}{(mb|mb')};
{\rm e}^{2\pi {\rm i} l/m} z^{1/m}m^{mp(-q)}\right],
\label{eq:A-product}
\end{gather}
here we have a summation, which is absent for the classical Meijer $G$-function.

Differentiating the integral by the parameter $z$ we get
\begin{subequations}
\begin{gather}
\left({-}z\frac\partial{\partial z}+a_j\right)\,
\Gc{p}{q}\left[\genfrac{}{}{0pt}{}{(a|a')}{(b|b')};z\right]
 =\Gc{p}{q}\left[\genfrac{}{}{0pt}{}{(a|a')_{\setminus j},(a_j+1)|a_j'}
{(b|b')};z\right],
\label{eq:A-diff1} \\
\left(z\frac\partial{\partial z}+b_m\right)\,
\Gc{p}{q}\left[\genfrac{}{}{0pt}{}{(a|a')}{(b|b')};z\right]
 =\Gc{p}{q}\left[\genfrac{}{}{0pt}{}{(a|a')}
{(b|b')_{\setminus m},(b_m+1)|b_m'};z\right].
\label{eq:A-diff2}
\end{gather}
\end{subequations}
Due to the $(-1)$ in \eqref{eq:shift1}, the similar equations for
$\ov z\frac{\partial}{\partial \ov z}$ differ from
\eqref{eq:A-diff1} and \eqref{eq:A-diff2} by signs
\begin{subequations}
\begin{gather}
\left({-}\ov z\frac\partial{\partial \ov z}+a_j'\right)\,
\Gc{p}{q}\left[\genfrac{}{}{0pt}{}{(a|a')}{(b|b')};z\right]
 =- \Gc{p}{q}\left[\genfrac{}{}{0pt}{}{(a|a')_{\setminus j},a_j|(a_j'+1)}
{(b|b')};z\right],
\label{eq:A-diff3} \\
\left(\ov z\frac\partial{\partial \ov z}+b'_m\right)\,
\Gc{p}{q}\left[\genfrac{}{}{0pt}{}{(a|a')}{(b|b')};z\right]
 =-
\Gc{p}{q}\left[\genfrac{}{}{0pt}{}{(a|a')}
{(b|b')_{\setminus m},b_m|(b_m'+1)};z\right].
\label{eq:A-diff4}
\end{gather}
\end{subequations}

The last four equalities \eqref{eq:A-diff1}--\eqref{eq:A-diff4}
imply the differential equations \eqref{eq:differential-eq}--\eqref{eq:ovcD}.

Also
\begin{gather*}
\Gc{p}{q}\left[\genfrac{}{}{0pt}{}{(a|a')_{\setminus j}, (a_j+1)|a_j'}
{(b|b')};z\right]
- \Gc{p}{q}\left[\genfrac{}{}{0pt}{}{(a|a')_{\setminus m}, (a_m+1)|a_m'}
{(b|b')};z\right] \\
\qquad{} =(a_j-a_m) \, \Gc{p}{q}\left[\genfrac{}{}{0pt}{}{(a|a')}{(b|b')};z\right],
\end{gather*}
and
\begin{gather*}
\Gc{p}{q}\left[\genfrac{}{}{0pt}{}{(a|a')_{\setminus j}, (a_j+1)|a_j'}
{(b|b')};z\right]
+ \Gc{p}{q}\left[\genfrac{}{}{0pt}{}{(a|a')}
{(b|b')_{\setminus m}, (b_m+1)|b_m'};z\right] \\
\qquad{} =(a_j+b_m) \,
\Gc{p}{q}\left[\genfrac{}{}{0pt}{}{(a|a')}{(b|b')};z\right].
\end{gather*}

\subsection{References, links and problems}\label{ss:generalities}

\textbf{(1).}
Gauss hypergeometric functions of the complex field (our $\Gc{2}{2}$
modulo a simple factor)
were defined by Gelfand, Graev and Retakh in~\cite{GGR}
by the Euler integral \eqref{eq:Euler}.
Many formulas for such functions were obtained in \cite[Section~3]{MN}.

 \textbf{(2).} Marichev \cite{Mar} proposed the following trick, which became the main tool
in the creation of the Prudnikov, Brychkov and Marichev
tables \cite{PBM},\footnote{See also the tables \cite{BMSa}.}
see also comments in \cite{PBM-VINITI}.
Many special functions (and many elementary functions)
are special cases of the Meijer $G$-functions, i.e.,
can be written as Mellin--Barnes integrals \eqref{eq:J}.
Therefore they are inverse Mellin transforms of products
\begin{equation}
\frac{\prod\limits_{\alpha=1}^A\Gamma(A_\alpha+\sigma)
\prod\limits_{\beta=1}^B\Gamma(B_\beta-\sigma)}
{\prod\limits_{\gamma=1}^C\Gamma(C_\gamma+\sigma)
\prod\limits_{\delta=1}^D\Gamma(D_\delta-\sigma)}.
\label{eq:product}
\end{equation}
Take two such functions $\Phi(x)$, $\Psi(x)$. Then
we can evaluate the convolution
\[
\Theta(x)=\int_0^\infty \Phi(y)\Psi(x/y)y^{-1} \,{\rm d}y.
\]
Indeed, the Mellin transform of $\Theta$ is the product of
Mellin transforms, therefore we get a product of two functions
of the type \eqref{eq:product}, i.e., a function
of the same type. Now we can express $\Theta(x)$ as a linear
combination of hypergeometric functions.

Numerous formulas in tables of integrals (such as Gradshteyn and
Ryzhik \cite{GR})
whose initial derivations were ingenious can be obtained in this
straightforward way.
The table of evaluations of $G$-functions in Prudnikov, Brychkov and
Marichev \cite[Vol.~3, Section~8.4]{PBM} contains 90 pages,
for each pair of lines we can apply this trick.\footnote{Chapter~7 of the
same book (160 pages) also provides us a material for this business.}

Our Theorem~\ref{cor:2} with formulas
\eqref{eq:symmetry1}, \eqref{eq:A-cancelation}, \eqref{eq:shift}
gives us the same tool.\footnote{Our arguments are not sufficient
for integrals \eqref{eq:convolutions} with functions
$\Gc{1}{0}$, see~\eqref{eq:1G0}. Apparently, the formula~\eqref{eq:convolutions} remains valid in this case.}
However, in our case the picture is less sophisticated.
The classical Meijer functions depend on $4$ subscripts and
superscripts (see~\eqref{eq:J}),
In our case the reflection formula~\eqref{eq:gamma-complement}
allows to move $\Gamma^\C$-factors from the denominator to numerator.
As a result, functions $\Gc{p}{q}$ depend only on two subscripts $p$, $q$.

Apparently, most\footnote{With some exceptions, for instance an application
of formula \eqref{eq:A-product} can be problematic.}
of identities for classical hypergeometric functions as
they are exposed in \cite{HTF1,HTF2} (Chapters 2, 4, 6, 7),
\cite{AAR, Sla} have counterparts for functions $\Gc{p}{q}$, but
different classical formulas can have the same counterpart
(for instance the ${_5H_5}$-Dougall formula and the de Branges--Wilson
integral correspond to one integral over $\Z\times \R$, see~\cite{Ner-doug}).

\textbf{(3).}
A collection of beta-integrals involving products
of $\Gamma^\C$ is known, see Bazhanov, Man\-ga\-zeev and Segeev~\cite{BMS},
Kels \cite{Kel1,Kel2}, Derkachov, Manashov and Valinevich
\cite{DM1,DM2,DMV1,DMV2},
Nere\-tin~\cite{Ner-doug}, Sarkissian and Spiridonov \cite{SS},
in particular, this collection contains counterparts of the
de Branges--Wilson integral and the Nassrallah--Rahman integral.
Such integrals can be regarded as evaluations
of functions $\Gc{p}{p}(z)$ at the point $(-1)^p$.

\textbf{(4).}
Such beta integrals and such hypergeometric functions
arise as limits of elliptic beta-integrals and the hypergeometric
functions $\Gc{p}{p}$ as limits of elliptic hypergeometric functions,
see Sarkissian and Spiridonov \cite{SS}.

\textbf{(5).}
The classical expansion in Jacobi polynomials has
a well-known continuous analog known under terms
`Olevski transform', `generalized Mehler--Fock transform',
`Jacobi transform', see, e.g., \cite{Koo};
there is also a second continuous analog \cite{Ner-Jacobi}.
The paper Molchanov and Neretin~\cite{MN} contains a
complex counterpart of these $3$ transformations
(expansions in the Jacobi polynomials
and two integral operators), it is a unitary integral transform
with $\Gc{2}{2}$-kernel
acting from a~certain weighted $L^2$ on $\C$ to a certain weighted~$L^2$ on $\Lambda\simeq\Z\times \R$.

The `Jacobi transform' is a representative of a big zoo of
hypergeometric integral transforms (see, e.g.,~\cite{Yak}),
it is natural to think that their counterparts
exist in $\Gc{p}{q}$-cases.
Integral transforms also can be applied as a tool for the production
of special-functional identities (clearly, several transforms were used
for the creation of \cite[Vol.~3]{PBM}, on possibilities
of the Jacobi transform, see \cite{Ner-w}).
An example of application of the $\Gc{2}{2}$-transform
is contained in~\cite{Ner-doug}.

\textbf{(6).}
It is well-known that representation theory of the group
$\SL(2,\R)$ is closely related to theory of hypergeometric functions
(as the Bessel functions, the confluent hypergeometric functions,
the Gauss hypergeometric functions, $\hyp{3}{2}(1)$, and the balanced
$\hyp{4}{3}(1)$). Application of harmonic analysis
related to the Lorentz group\footnote{See old works Gelfand, Graev and
Vilenkin \cite{GGV}, Naimark \cite{Nai1,Nai2,Nai3},
Gelfand and Graev \cite{GG} on $\SL(2,\C)$-harmonic analysis,
see also Derkachov, Korchemsky and Manashov \cite{DKM,DM-plus}.} $\SL(2,\C)$ to special functions are far not so
popular (at least among pure mathematicians).
However, if to ask such a question, then hypergeometric functions
of complex field come thick and fast.

--- a tensor product $\rho_1\otimes\rho_2$ of two irreducible unitary
representations of $\SL(2,\C)$ is a multiplicity free direct integral
(see Naimark \cite{Nai1,Nai2,Nai3}).
Therefore we can canonically decompose a triple tensor product
\[
(\rho_1\otimes\rho_2)\otimes \rho_3=\rho_1\otimes(\rho_2\otimes \rho_3)
\]
in two ways.
In the first case we decompose $\rho_1\otimes\rho_2$ and multiply each
component by $\rho_3$, in the second case we start from $\rho_2\otimes \rho_3$. Realizing this idea\footnote{It is difficult to extend this approach
to unitary representations of other groups since decompositions of
tensor products usually have multiplicities $>1$ (even for $\SL(2,\R)$).
Several multiplicity free cases were examined in
\cite{Gro1,Ism1,Ism2, Ism0,Ism3}.}
we get two explicit decompositions
of the same representation, the intertwining operator splits
into a direct integral of intertwining operators acting in isotypic components,
such operators can be regarded as counterparts
of Racah coefficients ($6j$-symbols).
Ismagilov \cite{Ism1,Ism2}
(see also Derkachov iand Spiridonov \cite{DS}) showed that such
`Racah operators' are integrals transforms whose kernels have a form
$\Gc{4}{4}(1)$.
The `Racah operators'
are $G^\C$-counterparts of expansions in Racah polynomials, expansions in
Wilson polynomials, and the `Wilson function transforms' defined by
Groenevelt \cite{Gro,Gro1}.

--- Recall that the Lorentz group $\SL(2,\C)$ is locally isomorphic to
the complex orthogonal group $\SO(3,\C)$. Consider the symmetric space
$\SO(3,\C)/\SO(2,\C)$, it can be regarded as a~quadric $x^2+y^2+z^2=1$ in $\C^3$, or the complexification of
the sphere $x^2+y^2+z^2=1$ in~$\R^3$. Under the complexification,
the Laplacian on the real sphere splits into two commuting
Laplacians, one is holomorphic, another is anti-holomorphic.
A question about their joint spectral decomposition in a space
of $\SO(2)$-invariant functions leads to
${_2G_2}$-transform considered in Molchanov and Neretin~\cite{MN}.

The shortest way of appearance of $\Gc{2}{2}$-functions is discussed
in Section~\ref{ss:Vilenkin} below.

\textbf{(7).}
Dotsenko and Fateev \cite{DF} obtained a complex version
of the Selberg integral; Derkachov, Manashov and Valinevich
\cite{DMV1,DMV2} obtained multi-dimensional beta-integrals
with products of $\Gamma^\C$-functions
(counterparts of the Gustafson's extension of the second Barnes lemma).

There arises a question about multi-dimensional symmetric
$\C$-counterparts of the Heckman--Opdam hypergeometric functions \cite{HO}.
The obvious candidates are spherical distributions on symmetric spaces
$G_\C/K_\C$, where $G_\C$ is a complex semisimple Lie group and $K_\C$
a complex symmetric subgroup.

On the other hand there are no reasons to hope
that in the multi-dimensional case complex spherical transforms are
unitary operators (unexpectedly, the radial parts of Laplace operators can
be non-commuting in the Nelson sense, see \cite[Theorem 1.1]{MN}).

\subsection[The Vilenkin model for $\SL(2,\C)$]{The Vilenkin model for $\boldsymbol{\SL(2,\C)}$}\label{ss:Vilenkin}
First, we modify notation.
Let us denote elements of $\Lambda_\C$ by bold letters,
denote $a|a'$ by $\bfa$, $1|1$ by $\1$, and $(-1)^{a-a'}$ by $(-1)^\bfa$.
Since $\Lambda\simeq \Z\times \R$, we can denote
\[
\int_\Lambda(\dots)=\sum_{k\in\Z} \int_\R (\dots).
\]

Now let us explain how the hypergeometric functions $\hypc{2}{1}$
arise from representations of the Lorentz group $\SL(2,\C)$.
Recall that this group can be realized as the group of all complex
$2\times 2$ matrices
$\big(\begin{smallmatrix} a&b\\c&d \end{smallmatrix}\big)$
with determinant $ad-bc=1$.
The \emph{principal series of representations} $T_\sigmA$ of this group
is parametrized by $\sigmA=\sigma|\sigma'\in\Lambda_\C$.
They act in the space of functions on $\C$ by operators
\[
T_\sigmA\begin{pmatrix} a&b\\c&d \end{pmatrix}
f(z):= f\left(\frac{b+zd}{a+zc}\right) (a+zc)^{-\1+\sigmA}.
\]
For $\sigma\in\Lambda$ we get unitary representations in $L^2(\C)$,
for details, see \cite[Chapter~III]{GGV}.

Let us realize the representations $T_\sigmA$ in a space of functions
on $\Lambda$ conjugating them by the Mellin transform\footnote{Cf.\
a model of representations of $\SL(2,\R)$ in Vilenkin
\cite[Section~VII.3]{Vil}.}
\[
\wt T_\sigmA\begin{pmatrix} a&b\\c&d \end{pmatrix} =
\cM \circ T_\sigmA\begin{pmatrix} a&b\\c&d \end{pmatrix}
\circ \cM^{-1}.
\]
A straightforward calculation shows that
\[
\wt T_\sigmA\begin{pmatrix} a&b\\c&d \end{pmatrix} F(\mu)
=\frac{1}{4\pi^2 {\rm i}} \int_\Lambda \cL\left[\mU,\lambdA;\begin{pmatrix}
 a&b\\c&d \end{pmatrix} \right] \,F(\lambdA)\,{\rm d}\lambdA,
\]
where
\[
\cL\left[\mU,\lambdA;\begin{pmatrix} a&b\\c&d \end{pmatrix} \right]=
\int_\C z^{\mU-1} (a+zc)^{\sigmA-\lambdA-1} (b+zd)^{\lambdA}\,\dd z.
\]
Substituting $z=-\frac ac u$ we come to
\[
 \cL\left[\mU,\lambdA;\begin{pmatrix} a&b\\c&d \end{pmatrix} \right]=
(-1)^\mU a^{\sigmA+\mU-\lambdA-\1} b^\lambdA c^{-\mU}\,
\frac{\Gamma^\C(\sigmA-\lambdA+\mU)}{\Gamma^\C(\mU)
\Gamma^\C(\sigmA-\lambdA)}\,\,
\hyp{2}{1}\left[\genfrac{}{}{0pt}{}{\mU,-\lambdA}{\sigmA-\lambdA+\mU};
\frac{ad}{bc}\right].
\]

\subsection*{Acknowledgements}
The work was supported by the grant FWF, P31591. I am grateful to \fbox{M.I.~Graev}, V.F.~Mol\-cha\-nov, V.A.~Spiridonov, and S.\'E.~Derkachov for discussions and references, I also thank the referees for useful comments.

\pdfbookmark[1]{References}{ref}
\LastPageEnding

\end{document}